\documentclass[11pt, reqno]{amsart}
\usepackage[all]{xy}

\usepackage{amssymb}
\usepackage{amsthm}
\usepackage{amsmath,cancel}
\usepackage{amscd}
\usepackage{verbatim}
\usepackage{float}
\usepackage{geometry}
\usepackage{color,mdframed}
\usepackage{bbm}
\usepackage[mathscr]{eucal}
\usepackage[all]{xy}
\usepackage{bm}

\usepackage{a4wide,fullpage}
\usepackage[draft]{hyperref}
\usepackage{mathtools}
\usepackage{calligra}
\usepackage{stmaryrd,float}
\DeclareMathAlphabet{\mathcalligra}{T1}{calligra}{m}{n}
\usepackage{enumitem}
\setitemize{leftmargin=*}

\theoremstyle{plain}
\newtheorem{theorem}{Theorem}
\newtheorem{proposition}[theorem]{Proposition}

\newtheorem{lemma}[theorem]{Lemma}
\newtheorem*{theorem*}{Theorem}

\theoremstyle{definition}

\newtheorem*{remark}{Remark}
\newtheorem*{remarks}{Remarks}
\theoremstyle{remark}

\newcommand{\R}{\mathbb{R}}

\newcommand{\Z}{\mathbb{Z}}
\newcommand{\N}{\mathbb{N}}
\def\C{\mathbb{C}}

\renewcommand{\H}{\mathbb{H}}

\newcommand{\z}{\mathfrak{z}}

\newcommand{\HH}{\mathbb{H}}

\newcommand{\Log}{\operatorname{Log}}

\newcommand{\re}{\operatorname{Re}}
\newcommand{\im}{\operatorname{Im}}

\newcommand{\SL}{{\text {\rm SL}}}
\newcommand{\PSL}{{\text {\rm PSL}}}

\newcommand{\Arg}{\operatorname{Arg}}
\newcommand{\Li}{\operatorname{Li}}

\def\lp{\left(}
\def\rp{\right)}

\numberwithin{equation}{section}
\numberwithin{theorem}{section}

\allowdisplaybreaks

\title[Meromorphic (generalized) \texorpdfstring{$L$}{L}-functions ]{Generalized \texorpdfstring{$L$}{L}-functions for meromorphic modular forms and their relation to the Riemann zeta function}
\author{Kathrin Bringmann}
\address{University of Cologne, Department of Mathematics and Computer Science, Weyertal 86-90, 50931 Cologne, Germany}
\email{kbringma@math.uni-koeln.de}
\author{Ben Kane}
\address{Department of Mathematics, University of Hong Kong, Pokfulam, Hong Kong}
\email{bkane@hku.hk}

\begin{document}
	
\begin{abstract}
In this paper, we construct a family of generalized $L$-functions, one for each point $z$ in the upper half-plane. We prove that as $z$ approaches $i\infty$, these generalized $L$-functions converge to an $L$-function which can be written in terms of the Riemann zeta function.
\end{abstract}
\date{\today}

\maketitle

\section{Introduction and statement of results}

We begin by recalling some basic properties of $L$-functions. They encode arithmetic information $c(p^r)$ that comes from local $p$-adic properties of some global object. This local information is encoded in a series 
\[
	L_p(s) := 1+\sum_{r=1}^\infty \frac{c\left(p^r\right)}{p^{rs}}.
\]
In many 
examples, one
 may write
\begin{equation}\label{eqn:localnice}
	L_p(s) = \frac{1}{1-f_p\left(p^{-s}\right)p^{-s}}
\end{equation}
for some polynomials $f_p$. Extending the definition of $c(p^r)$ multiplicatively, one that then forms the $L$-series as a Dirichlet series via the \begin{it}Euler product\end{it} 
\begin{equation}\label{eqn:EulerProduct}
L(s):=\prod_{p} L_p(s)=\sum_{n= 1}^\infty \frac{c(n)}{n^{s}},
\end{equation}
where the product runs over all primes and the series converges for $\re(s)$ sufficiently large if $c(n)$ grows at most polynomially in $n$. 

By adding archimedean information $L_{\infty}(s)$ at the real place one may form a complete $L$-function $\Lambda(s):=L_{\infty}(s)\prod_{p}L_p(s)$ which has a meromorphic continuation to  $\C$ and satisfies a functional equation
\begin{equation}\label{eqn:functional}
\Lambda(k-s)=\Lambda(s)
\end{equation}
for some $k\in\R$. The original example of an $L$-function is the \begin{it}Riemann zeta function\end{it}, defined for $s\in\C$ with $\re(s)>1$ as 
\begin{equation}\label{eqn:zetadef}
\zeta(s):=\sum_{n=1}^\infty \frac{1}{n^s}=\prod_{p} \frac{1}{1-p^{-s}}.
\end{equation}
 From \eqref{eqn:zetadef}, one sees both the Euler product and series representations characterized by \eqref{eqn:EulerProduct}, and the local factors indeed have the shape \eqref{eqn:localnice}. 
Riemann \cite{Riemann} proved that $\zeta$ has a meromorphic continuation to $\C$ with only a simple pole at $s=1$. He also showed that the function 
\[
\xi(s):=\frac{1}{2}\pi^{-\frac{s}{2}}s(s-1) \Gamma\left(\frac{s}{2}\right)\zeta(s),
\]
where $\Gamma(s)$ is the Gamma function, satisfies the functional equation 
\begin{equation}\label{func}
	\xi(s) = \xi(1-s),
\end{equation}
verifying the property \eqref{eqn:functional}.

 The functional equation \eqref{func} can be explained by viewing the zeta function as the Mellin transform of a modular form  (see \cite[p. 452]{ZagierMellin}). It is hence natural to ask whether the functional equation \eqref{eqn:functional} more generally arises from connections with modular objects. A number of results and important conjectures about the interplay between modular objects and $L$-functions have originated from this question.  For example, 
Hamburger \cite{Hamburger} proved a converse theorem characterizing $\zeta(s)$ by its functional equation which was later extended by Hecke \cite{HeckeConverse} and
 Weil \cite{WeilConverse} (see also \cite{BookerKrishnamurthy}). 
Specifically,
 the version of the converse theorem of Hecke states that if $c(n)$ are slow-growing and \eqref{eqn:functional} is satisfied, then the $L$-function comes from a Mellin transform of a modular form over $\SL_2(\Z)$. 
The
 numbers $c(n)$ are then realized as the Fourier coefficients of the corresponding modular 
form. 
 The idea to relate objects in different settings through their corresponding $L$-functions is one of the cornerstones of the Langlands program \cite{Langlands}, which conjectures deep connections between number theory and geometry.

In this paper, we construct a family of functions $L_z(s)$ for $s\in\C$ and $z\in\H:=\{z\in\C: \im(z)>0\}$ away from 
\[
\mathcal{S}:=\left\{z\in\H: \exists M\in\SL_2(\Z)\text{ such that } Mz\in i\R^+\right\}.
\]
We call these functions generalized $L$-functions because they resemble $L$-functions. By the converse theorem of Hecke \cite{HeckeConverse} and Weil \cite{WeilConverse}, if these are not Mellin transforms of modular forms, 
then either the Euler product and series representations in \eqref{eqn:EulerProduct} or the functional equation in \eqref{eqn:functional}
must not hold.   We consider the functional equation \eqref{eqn:functional} to be the fundamental property of the function in $s$  and then concentrate on the properties as a function of the other variable $z$. In particular, the (generalized) $L$-functions $L_z$ satisfy a functional equation, are harmonic and $\SL_2(\Z)$-invariant as a function of $z\in \H\setminus \mathcal{S}$, and are related to the Riemann zeta function in the limit $z\to x+i\infty$ with $x\in\R\setminus\Z$.

\begin{theorem}\label{thm:main}
Let $s\in \C \setminus \{1\}, z\in\H$. There exists  $L_z(s)$ satisfying the following properties:
\begin{enumerate}[leftmargin=*, label=\rm(\arabic*)]
\item For $\left(\begin{smallmatrix}a&b\\c&d\end{smallmatrix}\right)\in\SL_2(\Z)$ and $z\in \H\setminus\mathcal{S}$ 
 we have
\[
L_{\frac{az+b}{cz+d}}(s)=L_{z}(s).
\]
\item 
The function $z\mapsto L_z(s)$ is harmonic on $\H\setminus \mathcal{S}$.
\item 
We have the functional equation
\[
L_z(2-s)=-L_z(s).
\]
\item Suppose that $1< \re(s)<2$. Then for $x\in\R\setminus\Z$ we have 
\[
\hspace{.5cm}\lim_{y\to\infty}\left(L_{x+iy}(s) - \frac{2\pi i}{s} y^{s} - \frac{2\pi i }{s-2} y^{2-s}+2\Arg\left(1-e^{2\pi i x}\right)y^{s-1}\right)=-\frac{24i}{(2\pi)^{s-1}}\Gamma(s)  \zeta(s)\zeta(s-1),
\]
where $\Arg$ denotes the principal value of the argument.
\end{enumerate}
\end{theorem}

\begin{remarks}
	\ \begin{enumerate}[leftmargin=*,label=\rm(\arabic*)]
		\item Although the generalized $L$-functions $L_z(s)$ do not seem to satisfy the properties of classical $L$-functions (we expect that one can define a series via the integral in \eqref{eqn:LNz*gendef} below, but we do not study that in this paper), note that the limit in Theorem \ref{thm:main} (4) is the $L$-function associated to the weight two Eisenstein series for $\SL_2(\Z)$. The series representation for $\zeta(s)\zeta(s-1)$ may be found in \eqref{zetap}) and its Euler product follows from the right-hand side of \eqref{eqn:zetadef}).
		
		\item A more general version of Theorem \ref{thm:main} (4) holds for all $s\in\C$ (see Theorem \ref{thm:LFun-RiemannZeta}). 
	\end{enumerate}
\end{remarks}

The paper is organized as follows.  In Section \ref{sec:prelim}, we recall polar harmonic Maass forms and relate them to the resolvent kernel, as well as introducing some well-known useful functions and their properties. In Section \ref{sec:generalizeddef}, we define the functions $L_z(s)$ and prove Theorem \ref{thm:main} (1)--(3). We finally show Theorem \ref{thm:main} (4) in Section \ref{sec:limit}.

\section*{Acknowledgements}

The first author is supported by the Deutsche Forschungsgemeinschaft (DFG) Grant No. BR 4082/5-1. The research of the second author was supported by grants from the Research Grants Council of the Hong Kong SAR, China (project numbers HKU 17301317 and 17303618).

\section{Preliminaries}\label{sec:prelim}

\subsection{Special functions and their properties}\label{specfuncprop}

 We require certain special functions and their properties. For $s\in\C$ and $y>0$, we let $\Gamma(s,y):=\int_{y}^{\infty} e^{-t}t^{s-1}dt$ be the {\it incomplete gamma function}. By \cite[(8.11.1)--(8.11.3)]{NIST}, as $y\to\infty$ we have
\begin{equation}\label{eqn:Gammaexp}
\Gamma(s,y) = y^{s-1}e^{-y}\left(\sum_{\ell=0}^{N-1}(s-\ell)_{\ell}y^{-\ell} + O\left(y^{-N}\right)\right),
\end{equation}
where $(a)_{\ell}:=\prod_{j=0}^{\ell-1} (a+j)$ is the \begin{it}rising factorial\end{it}. For $y_1,y_2\in\R$ with $y_1y_2>0$, we also require the \begin{it}generalized incomplete gamma function\end{it} $\Gamma\left(s,y_1,y_2\right):=\int_{y_1}^{y_2}e^{-t}t^{s-1}dt$. For $\re(s)>0$, we have the relation $\Gamma(s,y_1,y_2)=\gamma(s,y_2)-\gamma(s,y_1)$ where  $\gamma(s,y):=\int_0^{y} e^{-t}t^{s-1} dt$. We also require the identity 
(see \cite[8.5.1]{NIST})
\begin{equation}\label{eqn:Gamma1F1}
\Gamma\left(s,y_1,y_2\right)=\frac{y_2^s}{s}{_1F_1}\left(s;s+1;-y_2\right)-\frac{y_1^s}{s}{_1F_1}\left(s;s+1;-y_1\right).
\end{equation}
Here ${_1F_1}(a;b;y)$ denotes the confluent hypergeometric function and for $y_j<0$, $y^s_j$ is defined through the principal branch of the logarithm.
 The asymptotic behaviour of the generalized incomplete gamma function as $y_1\to\infty$ or $y_2\to\infty$ may thus be obtained from the asymptotic behaviour of the ${_1F_1}$-function. 
Namely, as $y\to \infty$ we have (see \cite[13.7.2]{NIST}, where we note that $M(a,b,y)={_1F_1}(a;b;y)$) 
\begin{equation*}\label{eqn:1F1growth}
{_1F_1}(a;b;y)=\Gamma(b)\left(\frac{e^{y}y^{a-b}}{\Gamma(a)} + \frac{e^{-\pi i a} y^{-a}}{\Gamma(b-a)}\right) \left(1+O_{a,b}\left(y^{-1}\right)\right).
\end{equation*}
Assuming that $s\notin\Z$, \cite[13.7.1]{NIST} furthermore implies that for any $N\in\N_0$ 
\begin{equation}\label{eqn:1F1ss+1}
{_1F_1}\left(s;s+1;y\right)\sim s e^{y}y^{-1}\left(\sum_{j=0}^{N}(1-s)_j  y^{-j}+O_{s,N}\left(y^{-N-1}\right)\right).
\end{equation}

Recall that for $\ell\in\R$ the \begin{it}polylogarithm function\end{it} is defined for $|Z|<1$ by (see \cite[25.12.10]{NIST}) 
\[
\Li_{\ell}(Z) := \sum_{n=1}^\infty \frac{Z^n}{n^{\ell}}. 
\]
We frequently use the identity
\begin{equation}\label{eqn:Lieval}
\lim_{\varepsilon\to 0^+} \sum_{n=1}^{\infty} \frac{e^{-2\pi n(ix+\varepsilon y)}}{\left( n(1+\varepsilon)\right)^{\ell}}=\lim_{\varepsilon\to 0^+}\frac{\Li_{\ell}\left(e^{-2\pi (i x+\varepsilon y)}\right)}{(1+\varepsilon)^{\ell}}=\Li_{\ell}\left(e^{-2\pi i x}\right),
\end{equation}
where in the last equality we employ the fact that, since $x\notin\Z$, we avoid the branch cut of $\Li_{\ell}(z)$ along the positive real axis from $1$ to $\infty$ and hence the limit exists.

\subsection{Modular forms and polar harmonic Maass forms}

As usual, for $k\in\Z$, $\gamma=\left(\begin{smallmatrix}a&b\\c&d\end{smallmatrix}\right)\in\SL_2(\Z)$, and $f:\H\to\C$ we define the \begin{it}weight $k$ slash operator\end{it} by
\[
f\big|_{k}\gamma(\tau):=j(\gamma,\tau)^{-k} f\left(\frac{a\tau+b}{c\tau+d}\right),
\]
where $j(\gamma,\tau):=c\tau+d$. 
To describe certain modular objects, we
 require the growth of several functions towards points in $\H\cup\{i\infty\}$. For a non-holomorphic modular form $f$ and a point $z\in \H$, we say that $f$  \begin{it}exhibits the growth $g_z$ at $\tau=z$\end{it} if $f(\tau)-g_z(\tau)$ is bounded in an open neighborhood around $z$. We say that a singularity of $f$ at a point $z\in\H\cup\{i\infty\}$ has \begin{it}finite order\end{it} if the following holds:
\begin{enumerate}[leftmargin=*]
\item 
If $z\in\H$, then there exists $n\in\N_0$ such that $(\tau-z)^{n}f(\tau)$  is bounded for $\tau$ in a sufficiently small neighborhood of $z$.
\item 
If $z=i\infty$, then there exists $n\in\N_0$ such that $f(\tau)e^{-2\pi nv}$  is bounded for $v$ sufficiently large.
\end{enumerate}

 For $\tau=u+iv$, the weight $k$ \begin{it}hyperbolic Laplace operator\end{it} is given by
\[
\Delta_{k}:=-v^2\left(\frac{\partial^2}{\partial u^2}+\frac{\partial^2}{\partial v^2}\right)+ikv\left(\frac{\partial}{\partial u}+i\frac{\partial}{\partial v}\right).
\]
Throughout if we need to specify in which variable an operator $\mathcal{O}_k$ is taken, then we write $\mathcal{O}_{k,\tau}$. 
We call a function $f:\H\to\C$ a {\it weight $k$ polar harmonic Maass form} if it satisfies the following properties:
\begin{enumerate}[leftmargin=*, label=\rm(\arabic*)]
\item For every $\gamma\in\SL_2(\Z)$, we have
$
f|_k\gamma=f.
$
\item The function $f$ is annihilated by $\Delta_k$ except for a discrete set of singularities. 
\item The singularities of $f$ all have finite order.
\end{enumerate}

We require the following relations, which may be found for example in \cite[Lemma 5.2]{Book}, between the hyperbolic Laplace operator and the \begin{it}Maass raising operator\end{it} $R_{k}:=2i\frac{\partial}{\partial \tau} +\frac{k}{v}$.
\begin{lemma}\label{lem:RaiseLowerProp}
\ \begin{enumerate}[leftmargin=*, label=\rm(\arabic*)]
\item We have
\begin{equation*}
	R_k\left(f\big|_k\gamma\right) = R_k(f)\big|_{k+2}\gamma.
\end{equation*}
 In particular, if $f$ satisfies weight $k$ modularity, then $R_{k}(f)$ satisfies weight $k+2$ modularity.

\item If $\Delta_{k}(f)=\lambda f$, then $\Delta_{k+2}(R_k(f))=(\lambda+k)R_k(f)$.
\end{enumerate}
\end{lemma}

We require properties of some explicit modular functions (i.e., weight zero meromorphic modular forms). To describe these, for even $k\geq 4$ we define the {\it weight $k$ Eisenstein series} by
\[
	E_k(\tau) := 1-\frac{2k}{B_k} \sum_{n=1}^\infty \sigma_{k-1}(n) e^{2\pi i n\tau},
\]
where $\sigma_\ell(n):=\sum_{d\mid n} d^\ell$ and $B_k$ is the $k$-th Bernoulli number. We denote the unique normalized newform of weight $12$ on $\SL_2(\Z)$ by
\[
	\Delta(\tau):=q\prod_{n=1}^\infty (1-q^n)^{24}.
\]
We then set $J(\tau):=j(\tau)-744$, where $j(\tau):=\frac{E_4(\tau)^3}{\Delta(\tau)}$. We furthermore define 
the weight two meromorphic modular form
\begin{equation}\label{eqn:Hztaudef}
H_z(\tau):=\frac{E_4(\tau)^2E_6(\tau)}{\Delta(\tau)(J(\tau)-J(z))}.
\end{equation}
 We require the behaviour of $H_z(\tau)$ as $\tau\to i\infty$. To obtain this, we note that by computing the first term of their Fourier expansions, one easily sees that as $v\to\infty$, we have
\begin{equation}\label{eqn:GrowthJ}
J(\tau)=e^{-2\pi i \tau}+O\left(e^{-2\pi v}\right),\quad  \frac{E_4(\tau)^2E_6(\tau)}{\Delta(\tau)} = e^{-2\pi i\tau} +O(1). 
\end{equation}
We see in particular
that as $\tau\to i\infty$ (resp. $\tau\to 0$), $H_z(\tau)$ exhibits the growth $1$ (resp. $\frac{1}{\tau^2}$), 
applying the
 weight two modularity 
in the second case.

\subsection{Real-analytic Eisenstein series}

Setting $\Gamma_{\infty}:=\{ \pm \left(\begin{smallmatrix}1&n\\0&1\end{smallmatrix}\right): n\in\Z\}$, 
the \begin{it}weight $k$ real-analytic Eisenstein series\end{it} is defined for $k\in 2\N_0$ and $w\in\C$ with $\re(w)>1$ by
\begin{equation*}
E_{k}(w;\tau):=\sum_{\gamma\in\Gamma_{\infty}\backslash\SL_2(\Z)} v^w\big|_{k}\gamma
\end{equation*}
and then continued meromorphically to the complex $w$-plane (see \cite[Theorem 4.4.2]{Kubota}).  Since we make frequent use of results from \cite{GrossZagierSingular,GrossZagier}, we note for comparison the definition \cite[(2.14) on p. 239]{GrossZagier} specialized to $N=1$ (see also the equivalent definition in \cite[Section 5]{GrossZagierSingular}), where the weight $k=0$ is omitted, the order of the variables is flipped, and a semicolon is used to separate the variables instead of a comma.

We furthermore define the weight two harmonic Eisenstein series as
\[
\widehat{E}_2(\tau):=E_2(\tau)-\frac{3}{\pi v}\qquad\text{ where }E_2(\tau):=1-24\sum_{n=1}^{\infty} \sigma_1(n) e^{2\pi i n\tau}.
\]
Using the so-called Hecke trick, Hecke showed that 
\begin{equation*}
\widehat{E}_{2}(\tau)=  E_{2}(0;\tau).
\end{equation*}

The weight zero and weight two real-analytic Eisenstein series are related via the raising operator, as a direct calculation shows.
\begin{lemma}\label{lem:E0deriv}
For $w\in\C$ 
which is not a pole of $(w-1)E_0(w;z)$ and
 $k\in 2\N_0$, we have 
\[
R_{k}\left( E_{k}(w;\tau)\right)=(w+k)E_{k+2}(w-1;\tau).
\]
In particular, 
\[
R_{0}\left(E_{0}(1;\tau)\right)=\widehat{E}_{2}(\tau).
\]
\end{lemma}

Using the fact that $E_2(w;\tau)$ is an eigenfunction under $\Delta_2$, one obtains the following asymptotic behaviour of the Eisenstein series (see \cite[(2.17) on p. 240]{GrossZagier}).
\begin{lemma}\label{lem:E2growth}
Let $w\in\C$ be given such that $\Gamma(w)\zeta(2w-2)\neq 0$. Then, as $t\to \infty$,
\[
 E_2(w;it)= t^{w} - \frac{\sqrt{\pi}w\Gamma\left(w+\frac{1}{2}\right)\zeta(2w+1)}{\Gamma(w+2)\zeta(2w+2)} t^{-w-1} + O_w\left(e^{-t}\right).
\]
The dependence on $w$ in the error term is locally uniform. 
\end{lemma}

\subsection{Weight two polar harmonic Maass forms and the resolvent kernel}\label{sec:HzN}

In this subsection, we recall certain weight two polar harmonic Maass forms $H_{z}^*$ on $\SL_2(\Z)$. Explicitly, we define (writing $z=x+iy$ throughout)
\begin{equation*}
H_{z}^*(\tau):=-\frac{y}{2\pi} \Psi_{2}(\tau,z),
\end{equation*}
where $y \Psi_{2}(\tau,z)$ is the analytic continuation to $w=0$ of the Poincar\'e series (see \cite[Section 3.1]{BKweight0})
\begin{align*}
\mathcal{P}_{w}(\tau,z):=\sum_{M\in \SL_2(\Z)} \frac{\varphi_w(M\tau,z)}{j(M,\tau)^2|j(M,\tau)|^{2w}}.
\end{align*}
Here $\varphi_w(\tau,z):=y^{w+1}(\tau-z)^{-1}(\tau-\bar{z})^{-1}|\tau-\bar{z}|^{-2w}$. Setting  $X_{\tau}(z):=\frac{z-\tau}{z-\overline{\tau}}$, we use the following properties of $z\mapsto H_z^*$ which follow by \cite[Lemma 4.4 and Proposition 5.1]{BKweight0},
 and a direct calculation.
\begin{lemma}\label{lem:Hz*prop}
The function $z\mapsto H_{z}^*(\tau)$ is a weight zero polar harmonic Maass forms. Moreover, if $\tau\in\H$ is not an elliptic fixed point, then it exhibits the growth $-\frac{1}{4\pi v X_{\tau}(z)}$ at $z=\tau$ and no other 
singularities in $(\SL_2(\Z)\backslash {\H}) \cup \{i \infty\}$. In particular
\begin{equation}\label{eqn:HNz*Residue}
2\pi i \lim_{z\to\tau} (z-\tau) H_{z}^*(\tau) = 1.
\end{equation}
\end{lemma}

We also require the behaviour of $H_z^*$ as $z\to i\infty$ (see \cite[(1.7)]{BKLOR} and \cite[Theorem 1.2]{BKLOR}).

\begin{lemma}\label{lem:Hinftyprop}
We have
\begin{equation}\label{eqn:Hlimit}
\lim_{z\to i\infty}H_{z}^*=-\widehat{E}_{2}.
\end{equation}
\end{lemma}

We next relate $H_{z}^*$ to the resolvent kernel for $\SL_2(\Z)$ (see \cite{Hejhal} for a full treatment). 
The \begin{it}resolvent kernel\end{it} 
is defined by the analytic continuation in $w$ of 
\[
G_{w}(z,\tau):=\sum_{M\in\PSL_2(\Z)} g_{w}(Mz,\tau),
\]
where
\begin{equation*}\label{eqn:gsdef}
g_w(z,\tau):=-\frac{\Gamma(w)^2}{\Gamma(2w)}\left(\frac{2}{1+\cosh(d(z,\tau))}\right)^{w}{_2F_1}\left(w,w;2w;\frac{2}{1+\cosh(d(z,\tau))}\right).
\end{equation*}
Here ${_2F_1}(a,b;c;Z)$ is Gauss'
 hypergeometric function. Moreover, $d(z,\tau)$ is the hyperbolic distance between $z$ and $\tau$, which satisfies
\begin{equation*}
\cosh(d(z,\tau))=1+\frac{|z-\tau|^2}{2vy}.
\end{equation*}

The function $G_w(z,\tau)$ is invariant under $\SL_2(\Z)$ in both variables. Since $H_{z}^*(\tau)$ satisfies weight zero modularity in $z$ and weight two modularity in $\tau$ (see Lemma \ref{lem:Hz*prop} and \cite[Theorem 1.1]{BKLOR}), it is natural to apply the Maass raising operator in $\tau$. We therefore define
\begin{equation*}
\mathcal{G}_{w}(z,\tau):=\frac{1}{2i}R_{0,\tau}\left(G_{w}(z,\tau)\right).
\end{equation*}
We next use the invariance in both variables of $G_{w}(z,\tau)$ under $\SL_2(\Z)$  and the well-known facts that it is an eigenfunction under $\Delta_0$ in both variables (see \cite[property (b) in Section 5]{GrossZagierSingular}) and that for $\tau$ fixed it has a unique logarithmic singularity in $\SL_2(\Z)\backslash \H$ at $z=\tau$  (for example, see \cite[property (a) in Section 5]{GrossZagierSingular}). A direct calculation using 
Lemma \ref{lem:RaiseLowerProp}
 then yields the following.
\begin{lemma}\label{lem:calHmodular}
The function $z\mapsto \mathcal{G}_{w}(z,\tau)$ is $\SL_2(\Z)$-invariant and it is an eigenfunction with eigenvalue $w(1-w)$ under $\Delta_{0,z}$. The function $\tau\mapsto \mathcal{G}_{w}(z,\tau)$ satisfies weight two modularity and has eigenvalue $w(1-w)$ under $\Delta_{2,\tau}$. 
Moreover, for $\re(w)\geq 1$ and $\tau$ not an elliptic fixed point, $z\mapsto \mathcal{G}_w(z,\tau)$ exhibits the growth $\frac{1}{2i}R_{0,\tau}(g_w(z,\tau))$ at $z=\tau$, 
which simplifies in the special case $w=1$ to 
\begin{equation}\label{eqn:g1HResidue}
\lim_{z\to \tau} (z-\tau)\mathcal{G}_{1}(z,\tau)=-1,
\end{equation}
and does not grow at any point which is $\SL_2(\Z)$-inequivalent to $\tau$.
\end{lemma}

We additionally require the growth of $\mathcal{G}_w(z,\tau)$ as $z\to i\infty$ or $\tau\to i\infty$, which can be obtained from \cite[(6.5)]{Hejhal}. 
\begin{lemma}\label{lem:G1inftyprop}
Assume that $\re(w)\geq 1$.
\begin{enumerate}[leftmargin=*, label=\rm(\arabic*)]
\item For $y\ge v+\frac1v+\varepsilon$ 
with $\varepsilon>0$, we have 
\begin{equation*}
\mathcal{G}_{w}(z,\tau)= \frac{2\pi i}{2w-1} y^{1-w} R_{0,\tau}\left(E_{0}(w;\tau)\right) +O_{w,\varepsilon}\left(\left(v+\frac{1}{v}\right)^{\frac{1}{2}}e^{\frac{\pi}{2}\left(v+\frac1v-y\right)}\right),
\end{equation*}
where the error is locally uniform around $w=1$. In particular, as $y\to\infty$ 
\[
\mathcal{G}_{1}(z,\tau)=2\pi i \widehat{E}_{2}(\tau) +  O_{v}\left(e^{-\frac{\pi y}{2}}\right).
\]

\item 

For
 $v\ge y+\frac1y+\varepsilon$ with $\varepsilon>0$, we have  
\[
\mathcal{G}_w(z,\tau) = \frac{2\pi i (w-1)}{1-2w} v^{-w} E_0(w;z) +O_{w,\varepsilon}\left(\left(y+\frac{1}{y}\right)^{\frac{1}{2}}e^{\frac{\pi}{4}\left(y+\frac1y-v\right)}\right),
\]
where the bound is again locally uniform around $w=1$.
\end{enumerate}
\end{lemma}

The function $\mathcal{G}_{w}(z,\tau)$ is related to $H_{z}^*$ via the following proposition.
\begin{proposition}\label{prop:dGs}
We have
\begin{equation*}
\mathcal{G}_{1}(z,\tau) = -2\pi i H_{z}^*(\tau).
\end{equation*}
\end{proposition}

\begin{proof}
Lemma \ref{lem:Hz*prop} and Lemma \ref{lem:calHmodular} (with $w=1$) imply that 
\[
\mathbb{G}_\tau(z):=\lim_{\z\to z} \left(\mathcal{G}_{1}(\z,\tau)+2\pi i H^*_{\z}(\tau)\right),
\]
is a polar harmonic Maass form of weight zero on $\SL_2(\Z)$.  We claim that it vanishes identically. Without loss of generality, it suffices to assume that $\tau\in\H$ is not an elliptic fixed point. For $\mathfrak{z}\in\H$ which is not $\SL_2(\Z)$-equivalent to $\tau$, $\mathbb{G}_{\tau}$ does not have a singularity at $z=\mathfrak{z}$ because neither summand has a singularity by Lemmas \ref{lem:Hz*prop} and \ref{lem:calHmodular}. Moreover, \eqref{eqn:HNz*Residue} and \eqref{eqn:g1HResidue} yield that $\lim_{z\to\tau}(z-\tau)\mathbb{G}_\tau(z)=0.$
We conclude that $\mathbb{G}_{\tau}$ is an $\SL_2(\Z)$-invariant harmonic function that does not have any singularities.  Since the only weight zero harmonic functions on $\SL_2(\Z)$ without singularities are constant, we conclude that $\mathbb{G}_{\tau}(z)$ is independent of $z$ and $\lim_{z\to i\infty}\mathbb{G}_\tau(z)=0$ by Lemma \ref{lem:Hinftyprop} and Lemma \ref{lem:G1inftyprop} (1) implies that $\mathbb{G}_{\tau}(z)=0$ for all $z\in\H$, yielding the result. 
\end{proof}


\section{The definition of (generalized) \texorpdfstring{$L$}{L}-functions and the proof of Theorem \ref{thm:main} (1)--(3)}\label{sec:generalizeddef}

In this section we define the relevant (generalized) $L$-functions. In Section \ref{sec:E2L-Fun}, we relate an $L$-function to the weight two Eisenstein series. This $L$-function plays an important role in the proof of Theorem \ref{thm:main} (4) in Section \ref{sec:limit}. We then define $L_z(s)$ in Section \ref{sec:LFunDef} and investigate its main properties, proving Theorem \ref{thm:main} (1)--(3).

\subsection{An \texorpdfstring{$L$}{L}-function associated to the weight two Eisenstein series}\label{sec:E2L-Fun}
In order to relate $L_z(s)$ to the Riemann zeta function and prove Theorem \ref{thm:main} (4), we first 
recall the well-known construction of
 an $L$-function for the weight two Eisenstein series $E_2$ and its completion $\widehat{E}_2$. 
Following a trick of Riemann \cite{Riemann} (used to obtain the zeta function as a regularized Mellin transform), 
for $t_0>0$ we
 define 
\begin{multline}\label{eqn:L*E2def}
L\left(\widehat{E}_2,s\right):=\int_{t_0}^{\infty}\left(\widehat{E}_2(it)-1+\frac{3}{\pi t}\right)t^{s-1}dt\\
+\int_0^{t_0}\left(\widehat{E}_2(it)+\frac{1}{t^2}-\frac{3}{\pi t}\right)t^{s-1}dt-\frac{t_0^s}{s}-\frac{t_0^{s-2}}{s-2} + \frac{ 6}{\pi} \frac{t_0^{s-1}}{s-1}.
\end{multline}
We give the main properties of $L(\widehat{E}_2,s)$ and evaluate it in the following lemma, which may be easily proven using the modularity of $\widehat{E}_2$, the growth of $\widehat{E}_2(it)$ as $t\to\infty$, and the identity (see e.g. \cite[Theorem 291]{HardyWright})
\begin{equation}\label{zetap}
	\zeta(s)\zeta(s-\ell) =\sum_{n= 1}^\infty \frac{\sigma_{\ell}(n)}{n^s}.
\end{equation}

\begin{lemma}\label{lem:LE2eval}
\ \begin{enumerate}[leftmargin=*, label=\rm(\arabic*)]
\item The integrals on the right-hand side of \eqref{eqn:L*E2def} converge absolutely and define meromorphic functions on the complex $s$-plane with simple poles for $s\in\{0,1,2\}$.

\item The definition of $L(\widehat{E}_2,s)$ is independent of the choice of $t_0$.

\item We have
\[
L\left(\widehat{E}_2,s\right)=-\frac{24}{(2\pi)^{s}}\Gamma(s)\zeta(s)\zeta(s-1).
\]

\item We have
\[
L\left(\widehat{E}_2,2-s\right)=-L\left(\widehat{E}_2,s\right).
\]
\end{enumerate}
\end{lemma}

\subsection{Definition of a generalized $L$-function for polar harmonic Maass forms}\label{sec:LFunDef}
The goal of this section is to define $L_z(s)$ and to prove Theorem \ref{thm:main} (1)--(3). 

For $s,s_0,w\in\C$ with $\re(s),\re(w)$ sufficiently large, we set
\begin{equation}\label{eqn:LNz*gendef}
L_{z}(w,s_0;s):=\int_{0}^{\infty} \mathcal{G}_{w}(z,it)r_{z}(it)^{s_0}r_{z}\left(\frac{i}{t}\right)^{s_0} t^{s-1} dt,\qquad
r_{z}(\tau):=|X_{z}(\tau)|.
\end{equation}

\begin{lemma}\label{lem:absconverge}
Let $s_0,s\in\C$ and $z\in\H$ and if $z$ is equivalent under $\SL_2(\Z)$ to a point in $i\R^+$, then assume that $\re(s_0)$ is sufficiently large. Then the integral defining $L_{z}(w,s_0;s)$ converges absolutely and locally uniformly for $w\in\C$ with $\re(w)$ sufficiently large (depending on $\re(s)$).
\end{lemma}

\begin{proof}
Using the modularity from Lemma \ref{lem:calHmodular}, one may write (for $\re(w)$ sufficiently large)
\[
L_z(w,s_0;s)=\mathcal{J}_{w,s,s_0,z}(t_0)- \mathcal{J}_{w,2-s,s_0,z}\left(\frac{1}{t_0}\right)
\]
with 
\[
\mathcal{J}_{w,s,s_0,z}(t_0):=\int_{t_0}^{\infty} \mathcal{G}_{w}(z,it)r_{z}(it)^{s_0}r_{z}\left(\frac{i}{t}\right)^{s_0} t^{s-1} dt.
\]
The claim then follows by Lemma \ref{lem:G1inftyprop} (2).
\end{proof}

We next assume that $z\not\in \mathcal{S}$. In this case, we set 
$$L_z(w;s):=L_z(w,0;s).$$
\begin{theorem}\label{thm:Iwsprop}
	For each $z\in\H\setminus\mathcal{S}$ and $s\in\C$, the function $w\mapsto L_{z}(w;s)$ has a meromorphic continuation to the whole complex plane.  Furthermore, the resulting function in $z$ is an eigenfunction under $\Delta_{0,z}$ with eigenvalue $w(1-w)$ and it is invariant under the action of $\SL_2(\Z)$. Moreover, $L_z(w;s)$ satisfies the functional equation
	\[
		L_z(w;s) = -L_z(w;2-s).
	\] 
\end{theorem}

We prove Theorem \ref{thm:Iwsprop} through a series of lemmas and propositions. For $t_0>0$ we define
\begin{equation}\label{eqn:Iwsdef}
I_{w,s}(z):=\sum_{j=1}^2 \mathcal{I}_{j,w,s}(t_0;z)-\sum_{j=1}^{2}\mathcal{E}_{j,w,s}(t_0;z),
\end{equation}
where
\begin{align*}
\mathcal{I}_{1,w,s}(t_0;z)&:= \int_{0}^{t_0} \left(\mathcal{G}_{w}(z,it)+ \frac{ 2\pi i (w-1) }{1-2w} E_0(w;z)t^{w-2}\right) t^{s-1} dt,\\
\mathcal{I}_{2,w,s}(t_0;z)&:= \int_{t_0}^{\infty} \left(\mathcal{G}_{w}(z,it)-\frac{ 2\pi i (w-1) }{1-2w} E_0(w;z)t^{-w}\right) t^{s-1} dt,\\
\mathcal{E}_{1,w,s}(t_0;z)&:= \frac{2\pi i (w-1) t_0^{s+w-2}E_0(w;z)}{(1-2w)(s+w-2)},\\
\mathcal{E}_{2,w,s}(t_0;z)&:= \frac{2\pi i (w-1)t_0^{s-w}E_0(w;z)}{\left(1-2w\right)\left(s-w\right)}.
\end{align*}
We claim that $I_{w,s}(z)$ is independent of the choice of $t_0$ (see the remark after Lemma \ref{lem:ILagree}) and we show in Lemma \ref{lem:ILagree} below that it agrees with $L_z(w;s)$ for $\re(w)$ sufficiently large. Following this, we prove in Proposition \ref{prop:analyticcontinuation} that $I_{w,s}(z)$ gives a meromorphic continuation of $L_z(w;s)$ to the whole complex $w$-plane. We then show the functional equation of $I_{w,s}$ (and hence also of $L_z(w;s)$ by Lemma \ref{lem:ILagree}) in Proposition \ref{prop:Iwsfunctional}. 
The function $L_z(s)$ is then defined by the special value of this function at $w=1$ (see \eqref{eqn:Lz*def}). 
By using the absolute convergence of the integral in Lemma \ref{lem:absconverge} for $\re(w)$ sufficiently large and computing some elementary integrals, one easily sees that the functions $I_{w,s}(z)$ and $L_z(w;s)$ indeed coincide.

\begin{lemma}\label{lem:ILagree}
Let $s\in\C$ and $z\in\HH\setminus\mathcal{S}$. For $\re(w)$ sufficiently large, $I_{w,s}(z)=L_{z}(w;s)$. 
\end{lemma}
\begin{remark}
Note that Lemma \ref{lem:ILagree} implies that $I_{w,s}$ is independent of the choice of $t_0$ for $\re(w)$ sufficiently large because it agrees with $L_z(w;s)$. The Identity Theorem implies that it is independent of $t_0$ for all $w$ for which its analytic continuation exists. 
\end{remark}

The next step of the proof of Theorem \ref{thm:Iwsprop} is to prove that $I_{w,s}(z)$ gives a meromorphic continuation of $L_{z}(w;s)$ to the entire complex $w$-plane.

\begin{proposition}\label{prop:analyticcontinuation}
For each $s\in\C$, $z\in\HH\setminus\mathcal{S}$, and $t_0>0$, the functions $w\mapsto \mathcal{I}_{j,w,s}(t_0;z)$ $(j\in \{1,2\})$ converge absolutely and locally uniformly for $s,w\in\C$ outside of a discrete set of singularities at $w=\frac{1}{2}$ and at the poles of $(w-1)E_0(w;z)$. Moreover, if $w\notin\{\frac12,s,2-s\}$ and $w$ is not a pole of $(w-1)E_0(w;z)$, then the function $z\mapsto I_{w,s}(z)$ is an eigenfunction under $\Delta_0$ with eigenvalue $w(1-w)$ on $\HH\setminus\mathcal{S}$ and it is invariant under the action of $\SL_2(\Z)$.
\end{proposition}
\begin{proof}
First assume that $w\neq \frac{1}{2}$ and $w$ is not a pole of $(w-1)E_{0}(w;z)$. By Lemma \ref{lem:G1inftyprop} (2), the part of the integral defining $\mathcal{I}_{2,w,s}(t_0;z)$
 with $t>y+\frac1y+\varepsilon$ converges
 absolutely. Similarly, using the weight two modularity of $\mathcal{G}_w$ 
from Lemma \ref{lem:calHmodular},
 Lemma \ref{lem:G1inftyprop} (2) implies that for $t<(y+\frac1y)^{-1}$ 
\[
\mathcal{G}_{w}(z,it)=\frac{1}{(it)^2}\mathcal{G}_w\left(z,\frac{i}{t}\right)= -\frac{2\pi i (w-1)}{1-2w}  E_0(w;z)t^{w-2}+O\left(t^{-2}y^{\frac{1}{2}}e^{c\left(y+\frac1y-\frac{1}{t}\right)}\right).
\]
Hence the integral in $\mathcal{I}_{1,w,s}(t_0;z)$ 
with  $t<(y+\frac1y+\varepsilon)^{-1}$ also
 converges absolutely. One sees directly that $w\mapsto \mathcal{E}_{j,w,s}(t_0;z)$ is meromorphic with possible poles for $w\in\{\frac{1}{2},s, 2-s\}$, and the poles of $(w-1)E_{0}(w;z)$. 

We next show that $I_{w,s}(z)$ is an eigenfunction under $\Delta_{0,z}$. Since the function $E_0(w;z)$ is an eigenfunction with eigenvalue  $w(1-w)$ under $\Delta_{0,z}$, we see that $z\mapsto \mathcal{E}_{j,w,s}(t_0;z)$ are eigenfunctions under $\Delta_{0,z}$ with eigenvalue $w(1-w)$.  
Moreover, since the integrals $I_{j,w,s}(t_0;z)$ are absolutely and locally uniformly convergent, we may take the operator $\Delta_{0,z}$ inside the integrals. Combining the fact that $E_{0}(w;z)$ is an eigenfunction under $\Delta_{0,z}$ with the fact that   $\mathcal{G}_w(z,it)$ is also an eigenfunction with the same eigenvalue by Lemma \ref{lem:calHmodular}, the resulting integrals are also eigenfunctions. Modular invariance in $z$ also follows directly from the modularity of $E_0(z,w)$ and $\mathcal{G}_{w}(z,it)$, 
using Lemma \ref{lem:calHmodular}.
\end{proof}

Using the modularity from Lemma \ref{lem:calHmodular}, implies the functional equation of $I_{w,s}(z)$ in the usual way.

\begin{proposition}\label{prop:Iwsfunctional}
Suppose that $z\in\HH\setminus\mathcal{S}$ and $w\in\C \setminus \{\frac 12\}$ is not a pole of $(w-1)E_0(w;z)$. Then for $s\notin\{w,2-w\}$ we have
\[
I_{w,2-s}(z)=-I_{w,s}(z).
\]
\end{proposition}

 Since $I_{w,s}(z)$ provides the analytic continuation  of $L_z(w;s)$ 
by Lemma \ref{lem:ILagree} and Proposition \ref{prop:analyticcontinuation}, we set
\begin{equation}\label{eqn:Lz*def}
L_{z}(s):=I_{1,s}(z).
\end{equation}

\begin{remark}
	Due to the connection between $H_z^*$ and $\widehat{E}_2$ in \eqref{eqn:Hlimit}, one may naively consider $L(\widehat{E}_2,s)$  as the generalized $L$-function at $z=i\infty$. Although one cannot legally interchange the limit with the integrals defining $I_{1,s}(z)$ (and hence $L_z(s)$) to make this connection rigorous, the relationship between these (generalized) $L$-functions is investigated in Section \ref{sec:limit}. 
\end{remark}

\begin{proof}[Proof of Theorem \ref{thm:Iwsprop}]
Combining Lemma \ref{lem:ILagree} with Propositions \ref{prop:analyticcontinuation} and \ref{prop:Iwsfunctional} yields Theorem \ref{thm:Iwsprop}.
\end{proof}

\begin{proof}[Proof of Theorem \ref{thm:main} \textnormal{(1)}--\textnormal{(3)}]
 This follows from \eqref{eqn:Lz*def} and Theorem \ref{thm:Iwsprop}.
\end{proof}

\section{Behavior of the generalized $L$-functions towards infinity and the proof of Theorem \ref{thm:main} (4)}\label{sec:limit}

The goal of this section is to investigate the growth of $L_z(s)$ as $z\to i\infty$, ultimately proving Theorem \ref{thm:main} (4). In particular, in Theorem \ref{thm:LFun-RiemannZeta} we obtain an expansion of the type 
\begin{equation}\label{eqn:Lypowexp1}
L_z(s)=2 \pi i L\left(\widehat{E}_2,s\right) + \sum_{\ell=0}^{\left\lfloor\re(s)\right\rfloor} c_{\ell,s}(x) y^{s-\ell} +  \sum_{\ell=0}^{\left\lfloor 2-\re(s)\right\rfloor} d_{\ell,s}(x) y^{2-s-\ell} + o(1). 
\end{equation}
Due to the functional equations for $L_z(s)$ and $L(\widehat{E}_2,s)$, an expansion of the type \eqref{eqn:Lypowexp1}, if it exists, has a further restricted shape. 
\begin{lemma}\label{lem:expansionrestrict}
An expansion of the type \eqref{eqn:Lypowexp1} exists if and only if 
\begin{equation}\label{eqn:Lypowexp}
L_z(s)=2\pi i L\left(\widehat{E}_2,s\right) + \sum_{\ell=0}^{\left\lfloor\re(s)\right\rfloor} c_{\ell,s}(x) y^{s-\ell} -  \sum_{\ell=0}^{\left\lfloor 2-\re(s)\right\rfloor} c_{\ell,2-s}(x) y^{2-s-\ell} + o(1).
\end{equation}
Moreover, such an expansion holds if and only if it holds for $\re(s)\geq 1$.
\end{lemma}
We next relate $L_z(s)$ to $L(\widehat{E}_2,s)$. 
For ease of notation, we define (with $0\leq y_1\leq y_2\leq \infty$) 
\begin{align*}
\mathbb{J}_{z,s,0}(y_1,y_2)&:=\int_{y_1}^{y_2}\left( H_z(it)+\frac{1}{t^2}\right)t^{s-1} dt,\\
\mathbb{J}_{z,s,i\infty}(y_1,y_2)&:=\int_{y_1}^{y_2}\left( H_z(it)-1\right)t^{s-1} dt,
\end{align*}
where the subscripts $0$ and $i\infty$ indicate that we subtract the main growth of $H_z(it)$, which is defined in \eqref{eqn:Hztaudef}, towards $0$ or $i\infty$, respectively. 
\begin{lemma}\label{lem:Lz*LE2}
Suppose that $z\in\HH\setminus\mathcal{S}$. Then for any $t_0>0$ we have 
\[
L_z(s)=2\pi i L\left(\widehat{E}_2,s\right) -2\pi i  \mathbb{J}_{z,s,0}(0,t_0)-2\pi i \mathbb{J}_{z,s,i\infty}(t_0,\infty)
+2\pi i \frac{t_0^{s}}{s} +2\pi i \frac{t_0^{s-2}}{s-2}.
\]
\end{lemma}
\begin{proof}
Recalling the definition \eqref{eqn:Lz*def} and the absolute and locally uniform convergence of $I_{w,s}(z)$ shown in Proposition \ref{prop:analyticcontinuation}, we may directly plug $w=1$ into \eqref{eqn:Iwsdef}. It is well-known that (see \cite[p. 239, before (2.14)]{GrossZagier}) 
\[
\lim_{w\to 1} (w-1)E_0(w;z)= \frac{3}{\pi}.
\]
Plugging this into the definition following \eqref{eqn:Iwsdef}, we see directly that 
\begin{equation*}
\mathcal{E}_{1,1,s}(t_0;z)=\mathcal{E}_{2,1,s}(t_0;z)=-\frac{6 i t_0^{s-1}}{s-1}.
\end{equation*}
Moreover, using Proposition \ref{prop:dGs}, we obtain
\begin{equation}\label{eqn:I1seval2}
\sum_{j=1}^2\mathcal{I}_{j,1,s}(t_0;z) = -2\pi i \int_{0}^{t_0} \left(H_{z}^*(it)+\frac{3}{\pi t}\right)t^{s-1} dt -2\pi i  \int_{t_0}^{\infty} \left(H_z^*(it)-\frac{ 3}{\pi t}\right) t^{s-1} dt.
\end{equation}
We then use an identity of Asai, Kaneko, and Ninomiya \cite[Theorem 3]{AsaiKanekoNinomiya} and the second remark following \cite[Theorem 1.1]{BKLOR} to rewrite 
\[
H_z^*(\tau)=H_z(\tau)-\widehat{E}_2(\tau) .
\]
Plugging this into \eqref{eqn:I1seval2} and recalling the definition \eqref{eqn:L*E2def} yields the claim.
\end{proof}

Setting
\[
C_{\ell,s}(x):=\begin{cases}
\frac{2\pi i}{s}&\text{if }\ell = 0,\\
4\pi i\frac{(1-s)_{\ell-1}}{(2\pi)^{\ell}} \re\left(\Li_{\ell}\left(e^{2\pi i x}\right)\right)&\text{if }2\leq \ell\leq \left\lfloor\re(s)\right\rfloor\text{ is even},\\
-4\pi    \frac{\vphantom{\frac{1}{1}}(1-s)_{\ell-1}}{(2\pi)^{\ell}} \im\left(\Li_{\ell}\left(e^{2\pi i x}\right)\right)&\text{if }1\leq \ell\leq \left\lfloor\re(s)\right\rfloor\text{ is odd},
\end{cases}
\]
the generalized $L$-function $L_z(s)=I_{1,s}(z)$ is related to the Riemann zeta function via the following theorem.
\begin{theorem}\label{thm:LFun-RiemannZeta}
An expansion of the type \eqref{eqn:Lypowexp} exists. More precisely, for fixed  $x\in\R\setminus\Z$ we have
\begin{align*}
\lim_{y\to\infty}\left(L_{x+iy}(s)-\sum_{\ell=0}^{\left\lfloor\re(s)\right\rfloor} C_{\ell,s}(x) y^{s-\ell} +  \sum_{\ell=0}^{\left\lfloor 2-\re(s)\right\rfloor} C_{\ell,2-s}(x) y^{2-s-\ell}\right) &= 2\pi i L\lp\widehat{E}_2,s\rp\\
&= -\frac{24i}{(2\pi)^{s-1}}\Gamma(s)  \zeta(s)\zeta(s-1).
\end{align*}
\end{theorem}

\begin{proof} 
 By the Identity Theorem, it suffices to prove the claim for  $s\notin\Z$ and $\re(s)\geq 1$. Lemma \ref{lem:Lz*LE2} and  Lemma \ref{lem:LE2eval} (3) then imply that the claim of Theorem \ref{thm:LFun-RiemannZeta} is equivalent to 
\begin{multline}\label{eqn:finaltoshow}
\lim_{y\to\infty} \left(\mathbb{J}_{z,s,0}\left(0,t_0\right)+\mathbb{J}_{z,s,i\infty}\left(t_0,\infty\right)-\frac{i}{2\pi}\hspace{-.1cm} \sum_{\ell=0}^{\left\lfloor\re(s)\right\rfloor} \hspace{-.2cm} C_{\ell,s}(x) y^{s-\ell} + \frac{i}{2\pi} \hspace{-.1cm} \sum_{\ell=0}^{\left\lfloor 2-\re(s)\right\rfloor} \hspace{-.2cm} C_{\ell,2-s}(x) y^{2-s-\ell}\right)\\
= \frac{t_0^{s}}{s} + \frac{t_0^{s-2}}{s-2}.
\end{multline}
We assume without loss of generality  that $\frac{2}{y}<t_0<\frac{y}{2}$ and further split the integrals inside the limit. We claim that, as $y\to\infty$,
\begin{align}
\label{eqn:maintoshow1}
\mathbb{J}_{z,s,0}\left(0,\frac1y \right)&= -\frac{1}{2\pi} \Li_1\left(e^{-2\pi i x}\right)y^{1-s} +o_{x,s}(1),\\
\label{eqn:maintoshow2}
\mathbb{J}_{z,s,0}\left(\frac 1y,t_0\right)&=\frac{t_0^{s-2}-y^{2-s}}{s-2}+\frac{1}{2\pi} \Li_1\left(e^{2\pi i x}\right)y^{1-s}+o_{x,s}(1),\\
\label{eqn:maintoshow3}
\mathbb{J}_{z,s,i\infty}\left(t_0,y\right)&=\frac{t_0^s-y^s}{s} -  \sum_{\ell=1}^{\left\lfloor\re(s)\right\rfloor}  \frac{(1-s)_{\ell-1}}{(2\pi)^{\ell}}  \Li_{\ell}\left(e^{2\pi i x}\right)y^{s-\ell} +o_{x,s}(1),\\
\label{eqn:maintoshow4}
\mathbb{J}_{z,s,i\infty}\left(y,\infty\right)&=\sum_{\ell=1}^{\left\lfloor\re(s)\right\rfloor} (-1)^{\ell+1} \frac{(1-s)_{\ell-1}}{(2\pi)^{\ell}}\Li_{\ell}\left(e^{-2\pi i x}\right)y^{s-\ell}+o_{x,s}(1).
\end{align}
Before proving these, note that \eqref{eqn:maintoshow1}, \eqref{eqn:maintoshow2}, \eqref{eqn:maintoshow3}, and \eqref{eqn:maintoshow4} imply \eqref{eqn:finaltoshow} because
\[
-\Li_{\ell}\left(e^{2\pi i x}\right) + (-1)^{\ell+1}\Li_{\ell}\left(e^{-2\pi i x}\right) =\begin{cases} -2\re\left(\Li_{\ell}\left(e^{2\pi i x}\right)\right)&\text{if $\ell$ is even,}\\ -2i\im\left(\Li_{\ell}\left(e^{2\pi i x}\right)\right)&\text{if $\ell$ is odd.}\end{cases}
\]

We next show \eqref{eqn:maintoshow1}. Changing $t\mapsto \frac{1}{t}$ and using the weight two modularity of $H_z(\tau)$, we have
\begin{equation}\label{eqn:int1flip}
\mathbb{J}_{z,s,0}\left(0,\frac 1y \right)=-\mathbb{J}_{z,2-s,i\infty}(y,\infty)=-\lim_{\varepsilon\to 0^+} \mathbb{J}_{z,2-s,i\infty}((1+\varepsilon)y,\infty).
\end{equation}
As $t>(1+\varepsilon)y$, for $y$ sufficiently large the asymptotics in \eqref{eqn:GrowthJ} imply that $|J(it)|>|J(z)|$ and hence, for all $s\in\C$, 
\begin{align}
\nonumber \mathbb{J}_{z,2-s,i\infty}((1+\varepsilon)y,\infty)&= \int_{(1+\varepsilon)y}^{\infty}\left(\sum_{j=1}^{\infty}e^{-2\pi j (iz+t)}+O_x\left(e^{-2\pi t}\right) \left|\sum_{j=0}^{\infty}e^{-2\pi j (iz+t)}\right|\right)t^{1-s}dt\\
&\label{eqn:intq-power}= \int_{(1+\varepsilon)y}^{\infty}\sum_{j=1}^{\infty}e^{-2\pi j (iz+t)}t^{1-s}dt+O_{x,s}\left(e^{-\pi y}\right).
\end{align}
By the Dominated Convergence Theorem,  one can interchange the integral and sum for the main term.  Thus the main term of \eqref{eqn:intq-power} gives a contribution to \eqref{eqn:int1flip} of 
\begin{align}\label{eqn:int1incGamma}
-\lim_{\varepsilon\to0^+} \sum_{j=1}^\infty e^{-2\pi ijz} \int_{(1+\varepsilon)y}^\infty e^{-2\pi jt}t^{1-s} dt 
&= -\lim_{\varepsilon\to0^+} \sum_{j=1}^\infty e^{-2\pi ijz}(2\pi j)^{s-2}  \Gamma\left(2-s,2\pi j(1+\varepsilon)y\right)
\\ \nonumber
&=-y^{1-s}\lim_{\varepsilon\to 0^+}\sum_{j=1}^{\infty} \frac{e^{-2\pi j(ix+\varepsilon y)}}{2\pi j(1+\varepsilon)} \left(1+O\left(j^{-1}y^{-1}\right)\right),
\end{align}
taking \eqref{eqn:Gammaexp} with $s\mapsto 2-s$, $y\mapsto 2\pi(1+\varepsilon)jy$, and $N=1$.
It is not hard to see that 
\[
	-y^{1-s} \lim_{\varepsilon\to0^+} \sum_{j=1}^\infty \frac{e^{-2\pi j(ix+\varepsilon y)}}{2\pi j(1+\varepsilon)} O\left(j^{-1}y^{-1}\right) \ll y^{-1} \to 0
\]
as $y\to\infty$. We then obtain \eqref{eqn:maintoshow1} by \eqref{eqn:Lieval}.

We next prove \eqref{eqn:maintoshow2}. We write 
\begin{equation}\label{eqn:finalfirst1/ytot0}
\mathbb{J}_{z,s,0}\left(\frac1y,t_0\right) = \int_{\frac1y}^{t_0} H_z(it) t^{s-1}dt + \frac{t_0^{s-2}-y^{2-s}}{s-2}.
\end{equation}
In the integral we make the change of variables $t\mapsto \frac1t$ and use the modularity of the integrand to obtain that
\[
\int_{\frac1y}^{t_0} H_z(it)t^{s-1} dt = \lim_{\varepsilon\to0^+} \int_{\frac{1}{t_0}}^{(1-\varepsilon)y} H_z(it) t^{1-s}dt.
\]
Using \eqref{eqn:GrowthJ}, a straightforward calculation shows that for $\frac{1}{t_0}<t<(1-\varepsilon)y$ and $x\notin \Z$ we have 
\[
\frac{1}{J(z)-J(it)}=\frac{e^{2\pi i z}}{1-e^{2\pi (t+iz)}}\left(1+O_x\left(e^{-2\pi(t+y)}\right)\right).
\]
Combining this with the second equation from \eqref{eqn:GrowthJ}, we obtain that the integral on the right-hand side of \eqref{eqn:finalfirst1/ytot0} equals 
\begin{multline}\label{eqn:J1/yexpand}
	\lim_{\varepsilon\to 0^+}  \int_{\frac{1}{t_0}}^{(1-\varepsilon)y} \left(e^{2\pi t} +O(1)\right)\frac{e^{2\pi i z}}{1-e^{2\pi (t+iz)}} \left(1+ O_x\left(e^{-2\pi (t+y)}\right)\right) t^{1-s}dt\\
	=\lim_{\varepsilon\to 0^+}  \int_{\frac{1}{t_0}}^{(1-\varepsilon)y}\frac{e^{2\pi (t+i z)}}{1-e^{2\pi (t+iz)}} \left(1+ O_x\left(e^{-2\pi t}\right)\right) t^{1-s}dt.
\end{multline}
The $O$-term in \eqref{eqn:J1/yexpand} vanishes as $y\to\infty$ due to exponential decay of the integrand. To evaluate the main term, we expand $\frac{e^{2\pi (t+iz)}}{1-e^{2\pi(t+iz)}}$ as a geometric series to write the main term as 
\begin{equation}\label{eqn:int1geometric}
\lim_{\varepsilon\to 0^+}  \int_{\frac{1}{t_0}}^{(1-\varepsilon)y}\frac{e^{2\pi (t+i z)}}{1-e^{2\pi (t+iz)}} t^{1-s}dt=\lim_{\varepsilon\to 0^+} \sum_{\ell=1}^{\infty}  \int_{\frac{1}{t_0}}^{(1-\varepsilon)y}e^{2\pi \ell (t+i z)}t^{1-s}dt.
\end{equation}
Taking $t\mapsto -\frac{t}{2\pi\ell}$ and then plugging in \eqref{eqn:Gamma1F1}  yields that the main term in \eqref{eqn:J1/yexpand} is 
\begin{align}\nonumber
\lim_{\varepsilon\to0^+} \sum_{\ell=1}^\infty \int_{\frac{1}{t_0}}^{(1-\varepsilon)y}e^{2\pi\ell(t+iz)}t^{1-s}dt &= \lim_{\varepsilon\to0^+} \sum_{\ell=1}^\infty e^{2\pi i\ell z} e^{\pi i (s-2)} (2\pi\ell)^{s-2} \Gamma\left(2-s,-\frac{2\pi\ell}{t_0},-2\pi\ell(1-\varepsilon)y\right)\\
\nonumber
&= \frac{y^{2-s}}{2-s} \lim_{\varepsilon\to0^+} \sum_{\ell=1}^\infty e^{2\pi i\ell z} {_1F_1}\left(2-s;3-s;2\pi\ell(1-\varepsilon)y\right)\\
\label{eqn:int1middle}
&\hspace{2.5cm}- \frac{t_0^{s-2}}{2-s} \sum_{\ell=1}^\infty e^{2\pi i\ell z} {_1F_1}\left(2-s;3-s;\frac{2\pi\ell}{t_0}\right).
\end{align}
Plugging in \eqref{eqn:1F1ss+1} with $s\mapsto 2-s$ and  $N=0$, we obtain that for $|(\ell+1)Z|\to\infty$ (recall that we assume $s\notin\Z$ above)
\begin{equation}\label{eqn:1F1asymptotic}
{_1F_1}\left(2-s;3-s;2\pi\ell Z\right)= (2-s)e^{2\pi\ell Z} \left(\frac{1}{2\pi\ell Z}+O_s\left(\frac{1}{\ell^2Z^2}\right)\right).
\end{equation}
Taking $Z=t_0^{-1}$, the second term in \eqref{eqn:int1middle} vanishes as $y\to\infty$. 

We now take $Z=(1-\varepsilon)y$ in \eqref{eqn:1F1asymptotic} and plug into the first term in \eqref{eqn:int1middle} to obtain that \eqref{eqn:int1middle} equals
\begin{equation}\label{eqn:int1middlefirst}
y^{2-s}\lim_{\varepsilon\to 0^+}\sum_{\ell=1}^{\infty}  e^{2\pi\ell \left(iz + (1-\varepsilon)y\right)}\left(\frac{1}{2\pi(1-\varepsilon)\ell y}+O_{s}\left(\frac{1}{\ell^2y^2}\right)\right). 
\end{equation}
The error term in \eqref{eqn:int1middlefirst} is absolutely convergent uniformly in $\varepsilon\geq 0$ and gives a contribution $\ll y^{-\re(s)}$, which vanishes because $\re(s)\geq 1$ by assumption. Plugging \eqref{eqn:Lieval} with $x\mapsto -x$ into the main term of \eqref{eqn:int1middlefirst} implies \eqref{eqn:maintoshow2}.

We next show \eqref{eqn:maintoshow4}. 
Noting that the left-hand side of \eqref{eqn:maintoshow4} is the negative of the second expression of \eqref{eqn:int1flip} with $s\mapsto 2-s$, we may plug \eqref{eqn:int1incGamma} with $s\mapsto 2-s$ into \eqref{eqn:intq-power} and then use \eqref{eqn:Gammaexp} to compute 
\begin{multline*}
\mathbb{J}_{z,s,i\infty}(y,\infty)= \sum_{\ell=0}^{\left\lfloor \re(s)\right\rfloor-1}(s-\ell)_{\ell}  y^{s-\ell-1}\\
\times \lim_{\varepsilon\to 0^+}\sum_{j=1}^{\infty} e^{-2\pi j (ix+\varepsilon y)}\left(\left( 2\pi j(1+\varepsilon)\right)^{-\ell-1}+O\left(j^{-1-\left\lfloor\re(s)\right\rfloor}y^{\ell-\left\lfloor\re(s)\right\rfloor}\right)\right)+o_{x,s}(1).
\end{multline*}
The error term converges absolutely and uniformly in $\varepsilon\geq 0$ and decays as $y\to\infty$. Hence
\begin{equation}\label{eqn:int2bigt}
\mathbb{J}_{z,s,i\infty}(y,\infty)
=\lim_{\varepsilon\to 0^+} \sum_{\ell=0}^{\left\lfloor \re(s)\right\rfloor-1}(s-\ell)_{\ell} y^{s-\ell-1}\sum_{j=1}^{\infty} \frac{e^{-2\pi j(ix+\varepsilon y)}}{\left( 2\pi j(1+\varepsilon)\right)^{\ell+1}}+o_{x,s}(1).
\end{equation}
Plugging \eqref{eqn:Lieval} into the main term in \eqref{eqn:int2bigt} yields that the main term in 
\eqref{eqn:int2bigt} equals 
\begin{equation*}
\sum_{\ell=0}^{\left\lfloor \re(s)\right\rfloor-1}(s-\ell)_{\ell}  \frac{1}{(2\pi)^{\ell+1}} \Li_{\ell+1}\left(e^{-2\pi i x}\right) y^{s-\ell-1}
=\sum_{\ell=1}^{\left\lfloor \re(s)\right\rfloor}  \frac{(s+1-\ell)_{\ell-1}}{(2\pi)^{\ell}} \Li_{\ell}\left(e^{-2\pi i x}\right) y^{s-\ell}.
\end{equation*}
To obtain \eqref{eqn:maintoshow4}, we evaluate $(s+1-\ell)_{\ell-1} = (-1)^{\ell+1}(1-s)_{\ell-1}$.

It remains to prove \eqref{eqn:maintoshow3}. By taking $t\mapsto \frac1t$ we have 
\[
\mathbb{J}_{z,s,i\infty}\left(t_0,y\right)=-\mathbb{J}_{z,2-s,0}\left(\frac{1}{y},\frac{1}{t_0}\right).
\]
Plugging \eqref{eqn:J1/yexpand} with $s\mapsto 2-s$ and $t_0\mapsto \frac{1}{t_0}$ into \eqref{eqn:finalfirst1/ytot0} 
 hence yields 
\begin{equation}\label{eqn:Jt0yexpand}
\mathbb{J}_{z,s,i\infty}\left(t_0,y\right)=  - \lim_{\varepsilon\to 0^+}  \int_{t_0}^{(1-\varepsilon)y}\frac{e^{2\pi (t+i z)}}{1-e^{2\pi (t+iz)}} \left(1+ O_x\left(e^{-2\pi t}\right)\right) t^{s-1}dt+\frac{t_0^{s}-y^s}{s}.
\end{equation}
As in \eqref{eqn:J1/yexpand}, the $O$-term in \eqref{eqn:Jt0yexpand} again vanishes as $y\to\infty$. 
Expanding the main term as done in \eqref{eqn:int1geometric}, we
 then plug in \eqref{eqn:int1middle} with $t_0\mapsto \frac{1}{t_0}$ and $s\mapsto2-s$ to obtain 
\begin{multline}\label{eqn:int2middle}
\lim_{\varepsilon\to 0^+}\int_{t_0}^{(1-\varepsilon)y} \frac{e^{2\pi(t+iz)}}{1-e^{2\pi(t+iz)}} t^{s-1}dt=\lim_{\varepsilon\to 0^+}\sum_{\ell=1}^{\infty}  \int_{t_0}^{(1-\varepsilon)y}e^{2\pi \ell (t+iz)} t^{s-1}dt\\
= \frac{y^s}{s} \lim_{\varepsilon\to0^+} \sum_{\ell=1}^\infty e^{2\pi i\ell z} {_1F_1}\left(s;s+1;2\pi\ell(1-\varepsilon)y\right)-\frac{t_0^s}{s} \sum_{\ell=1}^\infty e^{2\pi i\ell z} {_1F_1}\left(s;s+1;2\pi\ell t_0\right).
\end{multline}
We then take $s\mapsto 2-s$ and $Z=t_0$ in \eqref{eqn:1F1asymptotic} to see that the second term in \eqref{eqn:int2middle} vanishes as $y\to \infty$.

Plugging $Z=2\pi\ell (1-\varepsilon)y$ and $N=\left\lfloor\re(s)\right\rfloor$ into \eqref{eqn:1F1ss+1}, the first term in \eqref{eqn:int2middle} becomes
\begin{equation}\label{eqn:int2middleasymptotic}
y^{s}\lim_{\varepsilon\to 0^+}\sum_{\ell=1}^{\infty} e^{2\pi i\ell z} e^{2\pi \ell(1-\varepsilon)y} \left(\sum_{j=0}^{\left\lfloor\re(s)\right\rfloor-1} \frac{(1-s)_j}{  \left(2\pi\ell(1-\varepsilon)y\right)^{j+1}}+O_{s}\left((\ell(1-\varepsilon) y)^{-\left\lfloor\re(s)\right\rfloor-1}\right)\right).
\end{equation}
Since $|e^{2\pi i\ell z} e^{2\pi \ell (1-\varepsilon)y}|=e^{-2\pi \ell \varepsilon y}\leq 1$, the contribution from the error term may be bounded against a constant times
\[
y^{\re(s)-\left\lfloor\re(s)\right\rfloor-1} \sum_{\ell=1}^{\infty} \ell^{-\lfloor\re(s)\rfloor-1}.
\]
Since $-\left\lfloor\re(s)\right\rfloor-1\leq -2$ (using the assumption $\re(s)\geq 1$), the sum on $\ell$ converges absolutely and $\re(s)-\left\lfloor\re(s)\right\rfloor-1<0$ implies that the error term in \eqref{eqn:int2middleasymptotic} vanishes as $y\to\infty$. 

We then interchange the sum in $\ell$ and $j$ in the main term of \eqref{eqn:int2middleasymptotic} (noting the exponential decay in the sum on $\ell$) to rewrite \eqref{eqn:int2middleasymptotic} as
\[
\sum_{j=0}^{\left\lfloor\re(s)\right\rfloor-1} (1-s)_j y^{s-j-1} \lim_{\varepsilon\to 0^+}\sum_{\ell=1}^{\infty}\frac{e^{2\pi \ell(ix-\varepsilon y)}}{\left(2\pi\ell(1-\varepsilon)\right)^{j+1}}+o(1).
\]
Plugging in \eqref{eqn:Lieval} with $x\mapsto-x$ yields that the above equals
\[
 \sum_{j=0}^{ \left\lfloor\re(s)\right\rfloor-1} \frac{(1-s)_j}{(2\pi)^{j+1}} \Li_{j+1}\left(e^{2\pi i x}\right)y^{s-1-j}+o(1)= \sum_{j=1}^{\left\lfloor\re(s)\right\rfloor} \frac{(1-s)_{j-1}}{(2\pi)^{j}} \Li_{j}\left(e^{2\pi i x}\right)y^{s-j}+o(1),
\]
giving \eqref{eqn:maintoshow3} and completing the proof.
\end{proof}

As a corollary, we obtain Theorem \ref{thm:main} (4).

\begin{proof}[Proof of Theorem \ref{thm:main} \rm{(4)}]
For $1<\re(s)<2$, the only terms that occur in Theorem \ref{thm:LFun-RiemannZeta} are the terms $\ell=0$ and $\ell=1$ in the first sum and the $\ell=0$ term in the second sum. Hence in this case Theorem \ref{thm:LFun-RiemannZeta} states that  
\[
\lim_{y\to\infty}\left(L_{x+iy}(s) - \frac{2\pi i}{s} y^{s} - \frac{2\pi i }{s-2} y^{2-s}+2\im\left(\Li_1\left(e^{2\pi i x}\right)\right)y^{s-1}\right) =-\frac{24i}{(2\pi)^{s-1}}\Gamma(s)  \zeta(s)\zeta(s-1).
\]
The proof follows noting that
\[
\Li_1(z)=-\Log(1-z), \quad \im\left(\Log(1-z)\right) = \Arg(1-z). \qedhere
\]
\end{proof}

\end{document}